\documentclass[11pt,reqno]{amsart}

\title{Model sets, almost periodic patterns, uniform density and linear maps}

\author{Pierre-Antoine Guihéneuf}
\address{Universidade Federal Fluminense, Instituto de Matemática e Estatística, Rua Mário Santos Braga, 24020-140 Niteroi, RJ, Brasil}
\email{pierre-antoine.guiheneuf@math.u-psud.fr}

\thanks{The author acknowledges Yves Meyer for his great help during this work, François Béguin and support from MESR (France) and CAPES (Brazil).}

\subjclass[2010]{52C23, 52C07, 11H06}

\usepackage[latin1]{inputenc}
\usepackage[english, french]{babel}
\usepackage[T1]{fontenc}
\usepackage{amsfonts}
\usepackage{amsmath}
\usepackage{amssymb}
\usepackage{stmaryrd}
\usepackage{amsthm}
\usepackage{enumerate}
\usepackage{pgf,tikz}
\usetikzlibrary{decorations.pathreplacing, shapes.multipart, arrows, matrix}
\usepackage{caption}
\usepackage[pdftex,colorlinks=true,linkcolor=blue,citecolor=blue,urlcolor=blue]{hyperref}

\newtheorem{lemme}{Lemma}
\newtheorem{theoreme}[lemme]{Theorem}
\newtheorem{prop}[lemme]{Proposition}
\newtheorem{coro}[lemme]{Corollary}

\newtheorem{theorem}{Theorem}

\theoremstyle{definition}
\newtheorem{definition}[lemme]{Definition}
\newtheorem{notation}[lemme]{Notation}

\theoremstyle{remark}
\newtheorem{rem}[lemme]{Remark}

\newcommand{\B}{\mathcal{B}}

\newcommand{\N}{\mathbf{N}}

\newcommand{\R}{\mathbf{R}}

\newcommand{\Q}{\mathbf{Q}}
\newcommand{\Z}{\mathbf{Z}}
\newcommand{\varep}{\varepsilon}
\newcommand{\Leb}{\operatorname{Leb}}

\newcommand{\card}{\operatorname{Card}}

\newcommand{\Id}{\operatorname{Id}}
\newcommand{\im}{\operatorname{im}}
\newcommand{\Vect}{\operatorname{span}}
\newcommand{\Vol}{\operatorname{Vol}}

\makeatletter
\newenvironment{abstracts}{%
  \ifx\maketitle\relax
    \ClassWarning{\@classname}{Abstract should precede
      \protect\maketitle\space in AMS document classes; reported}%
  \fi
  \global\setbox\abstractbox=\vtop \bgroup
    \normalfont\Small
    \list{}{\labelwidth\z@
      \leftmargin3pc \rightmargin\leftmargin
      \listparindent\normalparindent \itemindent\z@
      \parsep\z@ \@plus\p@
      
      \itemsep\bigskipamount
    }%
}{%
  \endlist\egroup
  \ifx\@setabstract\relax \@setabstracta \fi
}

\newcommand{\abstractin}[1]{%
  \otherlanguage{#1}%
  \item[\hskip\labelsep\scshape\abstractname.]%
}
\makeatother

\begin{document}

\begin{abstracts}
\abstractin{english}
This article consists in two independent parts. In the first one, we investigate the geometric properties of almost periodicity of model sets (or cut-and-project sets, defined under the weakest hypotheses); in particular we show that they are almost periodic patterns and thus possess a uniform density. In the second part, we prove that the class of model sets and almost periodic patterns are stable under discretizations of linear maps.

\abstractin{french}
Cet article se compose de deux parties indépendantes. Dans la première, nous étudions les propriétés géométriques de presque périodicité des ensembles modèle (ou ensembles \emph{cut-and-project}, dans leur définition la plus générale); en particulier, on montre que ce sont des ensembles presque périodiques et que par conséquent ils admettent une densité uniforme. Dans la seconde partie, on montre que les classes d'ensembles modèle et d'ensembles presque périodiques sont stables par application de discrétisations d'applications linéaires.
\end{abstracts}

\selectlanguage{english}

\maketitle

\setcounter{tocdepth}{2}

\section*[Introduction]{Introduction}

The almost periodicity of a (discrete) set can be treated from various viewpoints. For example, one can consider the convolution of the Dirac measure on this set with some test function, and define the almost periodicity of the set by looking at the almost periodicity of the function resulting from this convolution (which can be, for example, Bohr almost periodic); see for example \cite{MR2869161, MR1798989, MR2876415}. Also, some considerations about harmonic analysis can lead to different notions of almost periodicity of discrete sets, see for instance \cite{MR1798989, MR2135448}.

In particular, model sets, introduced by Y.~Meyer in \cite{MR0485769}, give a lot of (and maybe, the most of) examples of quasicrystals. But the sets obtained by this cut-and-project method are of big interest not only for the study of quasicrystals (see \cite{MR791727}) but also in other various domains of mathematics (see for example \cite{Moody25}), like the theory of almost periodic tilings \cite{Baake}, harmonic analysis and number theory (with applications for example to Pisot and Salem numbers) \cite{MR0485769}, dynamics of substitution systems \cite{MR1633181}, analysis of computer roundoff errors \cite{Gui15a}, etc.

Here, we will consider the problem of almost periodicity from a geometric point of view: basically, a discrete set $\Gamma\subset\R^n$ will be said almost periodic in a certain sense if for every $\varep>0$ it has a ``big'' set of $\varep$-almost periods.

\begin{definition}\label{almperiod}
Consider a family of (pseudo-)distances $(D_R)_{R>0}$ on the discrete subsets of $\R^n$. A vector $v\in \R^n$ is an \emph{$\varep$-almost-period} for the set $\Gamma\subset\R^n$ if $\limsup_{R\to +\infty} D_R(\Gamma,\Gamma-v)<\varep$.
\end{definition}

Various geometric notions of almost periodicity then arise from various hypothesis made on almost periods. It depends on
\begin{enumerate}
\item\label{cond1} how we measure the size of the set of $\varep$-almost-periods: we can only ask it to be relatively dense (see Definition~\ref{reldense}), or require it to have some kind of almost periodicity;
\item\label{cond2} what (pseudo-)distances $D_R$ are considered; here we will choose (we denote $B_R = B(0,R)$)
\[D_R(\Gamma,\Gamma') = \frac{\card\big((\Gamma\Delta\Gamma')\cap B_R\big)}{\Vol(B_R)},\]
or if we want more uniformity
\begin{equation}\label{EqDR}
D_R^+(\Gamma,\Gamma') = \sup_{x\in\R^n}\frac{\card\big((\Gamma\Delta\Gamma')\cap B(x,R)\big)}{\Vol(B_R)};
\end{equation}
\item\label{cond3} if we require some uniformity (in $v$) of the convergence of the $\limsup$ in Definition~\ref{almperiod}, etc.
\end{enumerate}
\bigskip

In this paper, we will investigate two properties that can be expected to be implied by a notion of almost periodicity:
\begin{enumerate}
\item possessing a uniform density (see Definition~\ref{defunif});
\item being stable under discretizations of linear maps (see Definition~\ref{DefDiscrLin}).
\end{enumerate}

It will turn out that the geometric notions that are natural for studying points (1) and (2) do not coincide: we will define two different geometric notions of almost periodicity:
\begin{itemize}
\item almost periodic patterns (see Definition~\ref{DefAlmPer}) and
\item weakly almost periodic sets (see Definition~\ref{wap}),
\end{itemize}
the former being adapted to point (2) and the latter being adapted to point (1).

Our main goal will be to explore the links between these two notions of almost periodicity, and that of model set (in the general sense given by Definition~\ref{DefModel}). In particular, the geometric point of view about almost periodicity will allow us to prove that model sets possess a uniform density (see for example \cite{MR1636780} for this result under stronger hypotheses over model sets). The following theorem is a combination of Theorem~\ref{ModelAlmost}, Proposition~\ref{WeakAlmPer} and Proposition~\ref{limitexist}.

\begin{theorem}\label{theoA}
\[ \boxed{\begin{array}{c} \!\Gamma\text{ model}\!\\  \text{set} \end{array}} \implies
\boxed{\begin{array}{c} \Gamma\text{ almost}\\  \!\text{periodic pattern}\! \end{array}} \implies
\boxed{\begin{array}{c} \Gamma\text{ weakly}\\  \!\text{almost periodic}\! \end{array}} \implies
\boxed{\begin{array}{c} \Gamma\text{ has a}\\  \!\text{uniform density}\! \end{array}}\]
\end{theorem}

\bigskip

In more details, we will use the following definition of model set.

\begin{definition}\label{DefModel}
Let $\Lambda$ be a lattice of $\R^{m+n}$, $p_1$ and $p_2$ the projections of $\R^{m+n}$ on respectively $\R^m\times \{0\}_{\R^n}$ and $\{0\}_{\R^m} \times \R^n$, and $W$ a measurable subset of $\R^m$. The \emph{model set} modelled on the lattice $\Lambda$ and the \emph{window} $W$ is (see Figure~\ref{FigModel})
\[\Gamma = \big\{ p_2(\lambda)\mid \lambda\in\Lambda,\, p_1(\lambda)\in W \big\}.\]
\end{definition}

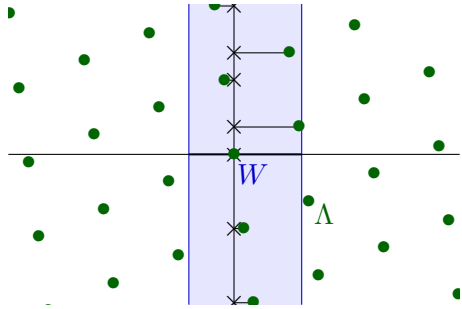
\begin{figure}[t]
\begin{center}
\begin{tikzpicture}[scale=1]
\fill[color=blue!10!white] (-.6,-2) rectangle (.9,2);
\draw[color=blue!80!black] (-.6,-2) -- (-.6,2);
\draw[color=blue!80!black] (.9,-2) -- (.9,2);
\draw[color=blue!80!black, thick] (-.6,0) -- (.9,0);
\draw[color=blue!80!black] (.25,0) node[below] {$W$};
\draw (-3,0) -- (3,0);
\draw (0,-2) -- (0,2);
\clip (-3,-2) rectangle (3,2);

\draw (.866,.364) -- (0,.364);
\draw (-.129,.987) -- (0,.987);
\draw (.737,1.351) -- (0,1.351);
\draw (-.258,1.974) -- (0,1.974);
\draw (.129,-.987) -- (0,-.987);
\draw (.258,-1.974) -- (0,-1.974);

\draw (0,0) node {$\times$};
\draw (0,.364) node {$\times$};
\draw (0,.987) node {$\times$};
\draw (0,1.351) node {$\times$};
\draw (0,1.974) node {$\times$};
\draw (0,-.987) node {$\times$};
\draw (0,-1.974) node {$\times$};

\foreach\i in {-3,...,3}{
\foreach\j in {-3,...,3}{
\draw[color=green!40!black] (.866*\i-.129*\j,.364*\i+.987*\j) node {$\bullet$};
}}
\draw[color=green!40!black] (1.2,-.8) node {$\Lambda$};
\end{tikzpicture}
\caption[Model set]{Construction of a model set (crosses) from a lattice (green dots) and a window $W$ (blue)}\label{FigModel}
\end{center}
\end{figure}

This definition of model set is close to that introduced by Y. Meyer in the early seventies \cite{MR0485769}, but more general: in our case the projection $p_2$ is not supposed to be injective. We will need this hypothesis to prove that the images of $\Z^n$ by discretizations of linear maps are model sets.

Model sets are sometimes called ``cut and project'' sets in the literature. Notice that their definition, which could seem very restrictive for the set $\Gamma$, is in fact quite general: as stated by Y. Meyer in \cite{MR0485769}, every Meyer set\footnote{A set $\Gamma$ is a \emph{Meyer set} if $\Gamma-\Gamma$ is a Delone set. It is equivalent to ask that there exists a finite set $F$ such that $\Gamma-\Gamma \subset \Gamma + F$ (see \cite{MR1400744}).} is a subset of a model set. Conversely, model sets are Meyer sets (see \cite{MR1420415}). 
\bigskip

The other notions of almost periodicity we will study concern Delone sets.

\begin{definition}\label{reldense}
Let $\Gamma$ be a subset of $\R^n$.
\begin{itemize}
\item We say that $\Gamma$ is \emph{relatively dense} if there exists $R_\Gamma>0$ such that each ball with radius at least $R_\Gamma$ contains at least one point of $\Gamma$.
\item We say that $\Gamma$ is a \emph{uniformly discrete} if there exists $r_\Gamma>0$ such that each ball with radius at most $r_\Gamma$ contains at most one point of $\Gamma$.
\end{itemize}
The set $\Gamma$ is called a \emph{Delone} set if it is both relatively dense and uniformly discrete.
\end{definition}

In some sense, the geometric behaviour of model sets is as regular as possible among non periodic sets. Indeed, we will prove that model sets with regular window are almost periodic patterns in the following sense (Theorem~\ref{ModelAlmost}).

\begin{definition}\label{DefAlmPer}\index{$\mathcal N_\varep$}
A Delone set $\Gamma$ is an \emph{almost periodic pattern} if for every $\varep>0$, there exists $R_\varep>0$ such that the set (recall that $D_R^+$ is defined by Equation~\eqref{EqDR})
\begin{equation}\label{EqAlmPer}
\mathcal N_\varep^{R_\varep} = \Big\{v\in\R^n \mid \forall R\ge R_\varep,\  D_R^+\big( (\Gamma+v)\Delta \Gamma \big) <\varep \Big\}.
\end{equation}
is relatively dense. We will denote $\mathcal N_\varep = \bigcup_{R'>0}\mathcal N_\varep^{R'}$
and call it the \emph{set of $\varep$-translations} of $\Gamma$.
\end{definition}

This definition of almost periodicity somehow requires as much uniformity as possible. Remark that in this definition, the assumption of almost periodicity of the set of $\varep$-almost-periods is the weakest possible: we only want it not to have big wholes. However, Proposition~\ref{azur} proves that in fact, these sets contain relatively dense almost periodic patterns ; Theorem~\ref{ModelAlmost} will also prove that for model sets, the sets of $\varep$-almost-periods contains relatively dense model sets. Remark that we do not know if the converse of Theorem~\ref{ModelAlmost} is true, i.e. if for any almost periodic pattern $\Gamma$ and any $\varep>0$ there exists a model set which coincides with $\Gamma$ up to a set of density smaller than $\varep$.

We then prove that almost periodic patterns possess a uniform density. This will be done by proving first that these sets are weakly almost periodic (Proposition~\ref{WeakAlmPer}; again, we do not know if the converse is true or not).

\begin{definition}\label{wap}
We say that a Delone set $\Gamma$ is \emph{weakly almost periodic} if for every $\varep>0$, there exists $R>0$ such that for every $x,y\in\R^n$, there exists $v\in\R^n$ such that
\begin{equation}\label{EqWeakAlmPer}
\frac{\card\Big( \Big(B(x,R)\cap\Gamma\Big) \Delta \Big(\big(B(y,R)\cap\Gamma\big)-v\Big) \Big)}{\Vol(B_R)} \le \varep.
\end{equation}
Remark that \emph{a priori}, the vector $v$ is different from $y-x$.
\end{definition}

This definition somehow requires the least assumptions about the pseudodistances $D_R$ (we do not consider supremum limits); it seems that this is the weakest definition that implies the existence of a uniform density (Proposition~\ref{limitexist}).

\begin{definition}\label{defunif}
A discrete set $\Gamma\subset \R^n$ possesses a \emph{uniform density} if there exists a number $D(\Gamma)$, called the \emph{uniform density}, such that for every $\varep>0$, there exists $R_\varep>0$ such that for every $R>R_\varep$ and every $x\in \R^n$,
\[\left| \frac{\card\big(B(x,R)\cap \Gamma\big)}{\Vol(B_R)} - D(\Gamma) \right| < \varep.\]
In particular, for every $x\in\R^n$, we have
\[D(\Gamma) = \lim_{R\to +\infty}\frac{\card\big(B(x,R)\cap \Gamma\big)}{\Vol(B_R)}.\]
\end{definition}

Remark that there exists some definitions of notions of almost periodicity that require less uniformity (e.g. \cite{MR2227822, GM}), however they do not imply the existence of a uniform density.

Note that the previous paper of the author with Y. Meyer \cite{GM} investigates the relations of these two notions of almost periodicity with some others, these others notions arising from considering the sum of Dirac measures on the discrete and considering properties of almost periodicity of these measures (as explained in the beginning of this introduction). Some examples of sets being almost periodic for one definition and not another are given. However, it is not proved that model sets, in the weak definition we use here, possess a uniform density.
\bigskip

In a second part, we will study the relations between these notions of almost periodicity and the discretizations of linear maps.

\begin{definition}\label{DefDiscrLin}
The map $P : \R\to\Z$\index{$P$} is defined as a projection from $\R$ onto $\Z$. More precisely, for $x\in\R$, $P(x)$ is the unique\footnote{Remark that the choice of where the inequality is strict and where it is not is arbitrary.} integer $k\in\Z$ such that $k-1/2 < x \le k + 1/2$. This projection induces the map\index{$\pi$}
\[\begin{array}{rrcl}
\pi : & \R^n & \longmapsto & \Z^n\\
 & (x_i)_{1\le i\le n} & \longmapsto & \big(P(x_i)\big)_{1\le i\le n}
\end{array}\]
which is an Euclidean projection on the lattice $\Z^n$. Let $A\in M_n(\R)$. We denote by $\widehat A$ the \emph{discretization}\index{$\widehat A$} of the linear map $A$, defined by 
\[\begin{array}{rrcl}
\widehat A : & \Z^n & \longrightarrow & \Z^n\\
 & x & \longmapsto & \pi(Ax).
\end{array}\]
\end{definition}

This definition can be used to model what happens when we apply a linear transformation to a numerical picture (see \cite{Gui15c, nouvel:tel-00444088, thibault:tel-00596947, Jacob25, MR1382839, MR1832794, A1996_1075}). These works mainly focus on the local behaviour of the images of $\Z^2$ by discretizations of linear maps: given a radius $R$, what pattern can follow the intersection of this set with any ball of radius $R$? What is the number of such patterns, what are their frequencies? Are these maps bijections? Also, the local behaviour of discretizations of diffeomorphism is described by the discretizations of linear maps. Thus, it is crucial to understand the dynamics of discretizations of linear maps to understand that of diffeomorphisms, for example from an ergodic viewpoint (see \cite{Gui15a, Gui15b} and the thesis \cite{Guih-These}).
\bigskip

We then obtain the following result (combination of Proposition~\ref{ImgModel} and Theorem~\ref{imgquasi}).

\begin{theorem}
The image of an almost periodic pattern by the discretization of a linear map is an almost periodic pattern. The image of a model set by the discretization of a linear map is a model set.
\end{theorem}

Unfortunately, the notion of weakly almost periodic set is not very convenient to manipulate and we have not succeeded to prove that it is stable under the action of discretization of linear maps. The reason of it is that we need some hypotheses of uniformity to prove the stability under discretizations of linear maps.
\bigskip

Let us summarize the notations we will use in this paper. We fix once for all an integer $n\ge 1$. We will denote by $\llbracket a, b \rrbracket$\index{$\llbracket\cdot\rrbracket$} the integer segment $[a,b]\cap\Z$. Every ball will be taken with respect to the infinite norm; in particular, for $x = (x_1,\cdots,x_n)$, we will have\index{$B(x,R)$}
\[B(x,R) = B_\infty(x,R) = \big\{y=(y_1,\cdots,y_n)\in\R^n\mid \forall i\in \llbracket 1, n\rrbracket, |x_i-y_i|<R\big\}.\]
We will also denote $B_R = B(0,R)$\index{$B_R$}. Finally, we will denote by $\lfloor x \rfloor$\index{$\lfloor \cdot \rfloor$} the biggest integer that is smaller than $x$ and $\lceil x \rceil$\index{$\lceil \cdot \rceil$} the smallest integer that is bigger than $x$.

\section{Model sets are almost periodic patterns}\label{SecModel}

In this section, we will only consider model sets whose window is regular.

\begin{definition}
Let $W$ be a subset of $\R^n$. We say that $W$ is \emph{regular} if for every affine subspace $V\subset\R^n$, we have
\[\Leb_V\Big( B_V\big(\partial (V\cap W),\eta\big)\Big) \underset{\eta\to 0}{\longrightarrow} 0,\]
where $\Leb_V$ denotes the Lebesgue measure on $V$, and $B_V\big(\partial (V\cap W),\eta\big)$ the set of points of $V$ whose distance to $\partial (V\cap W)$ is smaller than $\eta$ (of course, the boundary is also taken in restriction to $V$).
\end{definition}

\begin{theoreme}\label{ModelAlmost}
A model set modelled on a regular window is an almost periodic pattern. Moreover, for every $\varep>0$ the set of $\varep$-almost-periods contains a relatively dense model set. More precisely, for every $\varep>0$, there exists $R_\varep>0$ such that the set $\mathcal N_\varep^{R_\varep}$ (see Definition~\ref{DefAlmPer}) contains a relatively dense model set.
\end{theoreme}

Thus, model sets are in some sense the most regular as possible among non-periodic almost-periodic sets: for condition~(\ref{cond1}) page~\pageref{cond1}, the set of $\varep$-almost-periods is itself a model set, for condition~(\ref{cond2}) the distance on finite sets is the uniform distance, and for condition~(\ref{cond3}) we have uniformity in $v$ of convergence of the $\limsup$. Combined with Propositions \ref{WeakAlmPer} and \ref{limitexist}, this theorem directly implies the following corollary.

\begin{coro}\label{corolimitexist2}
Every model set possesses a uniform density.
\end{coro}

Remark that it seems to us that the simplest way to prove the convergence of the uniform density for model sets (in the generality of our definition) is to follow the strategy of this paper: prove that these model sets are weakly almost periodic and then that weakly almost periodic sets possess a uniform density.

The part of Theorem~\ref{ModelAlmost} stating that the sets of $\varep$-almost-periods contain relatively dense model sets should be qualified by Proposition~\ref{azur}: in general, the set of $\varep$-almost-periods is not only relatively dense but is also an almost periodic pattern.
\bigskip

To prove Theorem~\ref{ModelAlmost}, we begin by a weak version of it.

\begin{lemme}\label{LemModelAlmost}
A model set modelled on a window with nonempty interior is relatively dense.
\end{lemme}

\begin{proof}[Proof of Lemma \ref{LemModelAlmost}]
We prove this lemma in the specific case where the window is $B_\eta$ (recall that $B_\eta$ is the infinite ball of radius $\eta$ centred at 0). We will use this lemma only in this case (and the general case can be treated the same way).

Let $\Gamma$ be a model set modelled on a lattice $\Lambda$ and a window $B_\eta$. We will use the fact that for any centrally symmetric convex set $S\subset \R^n$, if there exists a basis $e_1,\cdots,e_n$ of $\Lambda$ such that for each $i$, $\lceil n/2 \rceil e_i \in S$, then $S$ contains a fundamental domain of $\R^n/\Lambda$, that is to say, for every $v\in\R^n$, we have $(S+v)\cap \Lambda\neq\emptyset$. This is due to the fact that the parallelepiped spanned by the vectors $e_i$ is included into the simplex spanned by the vectors $\lceil n/2\rceil e_i$.

We set
\[ V = \bigcap_{\eta'>0} \Vect\big(p_1^{-1}(B_{\eta'})\cap \Lambda\big) = \bigcap_{\eta'>0}\Vect\big\{\lambda\in\Lambda\mid d_\infty(\lambda,\ker p_1)\le\eta'\big\},\]
and remark that $\im p_2 = \ker p_1\subset V$, simply because for every vectorial line $D\subset \R^n$ (and in particular for $D\subset \ker p_1$), there exists some points of $\Lambda\setminus\{0\}$ arbitrarily close to $D$. We also take $R>0$ such that
\[V\subset V'\doteq\Vect\big(p_1^{-1}(B_{\eta/\lceil n/2\rceil})\cap \Lambda \cap p_2^{-1}(B_R)\big).\]
We then use the remark made in the beginning of this proof and apply it to the linear space $V'$, the set $S = \big(p_2^{-1}(B_R)\times p_1^{-1}(B_\eta)\big)\cap V'$, and the module $V'\cap \Lambda$. This leads to
\[\forall v\in V,\ \Big(\big(p_1^{-1}(B_\eta) \cap p_2^{-1}(B_R)\big) + v \Big) \cap \Lambda \neq \emptyset,\]
and as $\im p_2 \subset V$, we get
\[\forall v'\in \R^n,\ \big(p_1^{-1}(B_\eta) \cap p_2^{-1}(B_R+v')\big) \cap \Lambda \neq \emptyset;\]
this proves that the model set is relatively dense for the radius $R$.
\end{proof}

\begin{proof}[Proof of Theorem \ref{ModelAlmost}]
Let $\Gamma$ be a model set modelled on a lattice $\Lambda$ and a window $W$.

First of all, we decompose $\Lambda$ into three supplementary modules: $\Lambda = \Lambda_1 \oplus \Lambda_2 \oplus \Lambda_3$, such that (see \cite[Chap. VII, \S 1, 2]{MR1726872}):
\begin{enumerate}
\item $\Lambda_1 = \ker p_1 \cap \Lambda$;
\item $p_1(\Lambda_2)$ is discrete;
\item $p_1(\Lambda_3)$ is dense in the vector space $V$ it spans (and such a vector space is unique), and $V\cap p_1(\Lambda_2) = \{0\}$.
\end{enumerate}
As $\Lambda_1 = \ker p_1 \cap \Lambda = \im p_2 \cap \Lambda$, we have $\Lambda_1 = p_2(\Lambda_1)$. Thus, for every $\lambda_1\in\Lambda_1$ and every $\gamma\in\Gamma$, we have $\lambda_1+\gamma\in\Gamma$. So $\Lambda_1$ is a set of periods for $\Gamma$. Therefore, considering the quotients $\R^n/\Vect\Lambda_1$ and $\Lambda/\Lambda_1$ if necessary, we can suppose that ${p_1}_{|\Lambda}$ is injective (in other words, $\Lambda_1 = \{0\}$).

Under this assumption, the set $p_2(\Lambda_3)$ spans the whole space $\im p_2$. Indeed, as $\ker p_1\cap\Lambda = \{0\}$, we have the decomposition
\begin{equation}\label{eqAgr}
\R^{m+n} = \underbrace{\ker p_1}_{=\im p_2} \oplus \underbrace{\Vect\big(p_1(\Lambda_2)\big) \oplus \Vect\big(p_1(\Lambda_3)\big)}_{=\im p_1}.
\end{equation}
Remark that as $p_1(\Lambda_2)$ is discrete and $\ker p_1\cap\Lambda_2 = \{0\}$, we have $\dim\Vect\big(p_1(\Lambda_2)\big) = \dim\Lambda_2$; thus, considering the dimensions in the decomposition \eqref{eqAgr}, we get
\begin{equation}\label{EstimDimpt}
\dim\Vect(\Lambda_3) = \dim \Big(\ker p_1 \oplus \Vect\big(p_1(\Lambda_3)\big)\Big).
\end{equation}
The following matrix represents a basis of $\Lambda = \Lambda_2\oplus\Lambda_3$ in a basis adapted to the decomposition \eqref{eqAgr}.
\[\begin{array}{cc}
&
\begin{array}{cc}
\Lambda_2 &  \Lambda_3\\
\overbrace{\phantom{\hspace{1cm}}} & \overbrace{\phantom{\hspace{1cm}}}\end{array}\\
\begin{array}{r}
\ker p_1 = \im p_2           \Big\{\!\!\!\!\!\!\\
\Vect\big(p_1(\Lambda_2)\big)\Big\{\!\!\!\!\!\!\\
\Vect\big(p_1(\Lambda_3)\big)\Big\{\!\!\!\!\!\!
\end{array}
& \renewcommand{\arraystretch}{1.5}
\left(\begin{array}{c|c}
* & *\\ \hline
* & 0\\ \hline
\hspace{.4cm}0 \hspace{.4cm}{} & \hspace{.4cm} *\hspace{.4cm}{}
\end{array}\right)
\renewcommand{\arraystretch}{1}
\end{array}\]

We can see that the projection of the basis of $\Lambda_3$ on $\im p_2 \oplus \Vect\big(p_1(\Lambda_3)\big)$ form a free family; by Equation~\eqref{EstimDimpt}, this is in fact a basis. Thus, $\Vect(\Lambda_3)\supset \ker p_1 = \im p_2$, so $\Vect\big(p_2(\Lambda_3)\big) = \im (p_2)$.
\bigskip

For $\eta>0$, let $\mathcal N(\eta)$ be the model set modelled on $\Lambda$ and $B(0,\eta)$, that is
\[\mathcal N(\eta) = \{p_2(\lambda_3)\mid \lambda_3 \in \Lambda_3,\,\|p_1(\lambda_3)\|_\infty\le \eta\}.\]
Lemma \ref{LemModelAlmost} asserts that $\mathcal N(\eta)$ is relatively dense in the space it spans, and the previous paragraph asserts that this space is $\im p_2$. The next lemma, which obviously implies Theorem \ref{ModelAlmost}, expresses that if $\eta$ is small enough, then $\mathcal N(\eta)$ is the set of translations we look for.

\begin{lemme}\label{Lem2ModelAlmost}
For every $\varep>0$, there exists $\eta>0$ and a regular model set $Q(\eta)$ such that $D^+(Q(\eta))\leq \varep$ and
\[v\in \mathcal N(\eta)\Rightarrow (\Gamma+v)\Delta \Gamma\subset Q(\eta).\]
\end{lemme}

We have now reduced the proof of Theorem~\ref{ModelAlmost} to that of Lemma \ref{Lem2ModelAlmost}.
\end{proof}

\begin{proof}[Proof of Lemma \ref{Lem2ModelAlmost}]
We begin by proving that $(\Gamma+v)\setminus\Gamma\subset Q(\eta)$ when $v\in \mathcal N(\eta)$. As $v\in \mathcal N(\eta)$, there exists $\lambda_0\in \Lambda_3$ such that $p_2(\lambda_0)=v$ and $\|p_1(\lambda_0)\|_\infty\le \eta$.

If $x\in \Gamma+v$, then $x=p_2(\lambda_2+\lambda_3)+p_2(\lambda_0)=p_2(\lambda_2+\lambda_3+\lambda_0)$ where $\lambda_2 \in \Lambda_2$, $\lambda_3 \in \Lambda_3$ and $p_1(\lambda_2+\lambda_3)\in W$. If moreover $x\notin \Gamma$, it implies that $p_1(\lambda_2+\lambda_3+\lambda_0)\notin W$. Thus, $p_1(\lambda_2+\lambda_3+\lambda_0)\in W_\eta$, where (recall that $V = \Vect (p_1(\Lambda_3))$)
\[W_\eta = \big\{k+w\mid k\in\partial W, w\in V\cap B_\eta\big\}.\]
We have proved that $(\Gamma+v)\setminus\Gamma\subset Q(\eta)$, where
\[Q(\eta) =\big\{p_2(\lambda)\mid \lambda\in \Lambda,\,p_1(\lambda)\in W_\eta\big\}.\]
Let us stress that the model set $Q(\eta)$ does not depend on $v$. We now observe that as $W$ is regular, we have
\[\sum_{\lambda_2\in\Lambda_2} \Leb_{V+p_1(\lambda_2)}\big(W_\eta\cap (V+p_1(\lambda_2))\big)\underset{\eta\to 0}{\longrightarrow}0.\]
As $p_1(\Lambda_3)$ is dense in $V$ (thus, it is equidistributed), the uniform upper density of the model set $Q(\eta)$ defined by the window $W_\eta$ can be made smaller than $\varep$ by taking $\eta$ small enough.

The treatment of $\Gamma\setminus(\Gamma+v)$ is similar; this ends the proof of Lemma~\ref{Lem2ModelAlmost}.
\end{proof}

\section{Almost periodic patterns are weakly almost periodic sets}

In this section, we prove that almost periodic pattern are weakly almost periodic.

\begin{prop}\label{WeakAlmPer}
Every almost periodic pattern is weakly almost periodic.
\end{prop}

We do not know if the converse is true or not (that is, if there exists weakly almost periodic sets that are not almost periodic patterns). See also the addendum \cite{GM} of \cite{MR2876415} for more details on the subject.

\begin{proof}[Proof of Proposition \ref{WeakAlmPer}]
We prove that an almost periodic pattern satisfies Equation \eqref{EqWeakAlmPer} for $x=0$, the general case being obtained by applying this result twice.

Let $\Gamma$ be an almost periodic pattern and $\varep>0$. Then by definition, there exists $R_\varep>0$ and a relatively dense set $\mathcal N_\varep$ (for a parameter $R_{\mathcal N_\varep}>0$) such that
\begin{equation}\label{eq2.1}
\forall R\ge R_\varep,\  \forall v\in\mathcal N_\varep,\  D_{R}^+\big( (\Gamma+v)\Delta \Gamma \big) <\varep.
\end{equation}
Moreover, as $\Gamma$ is Delone, there exists $r_\Gamma>0$ such that each ball with radius smaller than $r_\Gamma$ contains at most one point of $\Gamma$.

\begin{figure}[b]
\begin{center}
\begin{tikzpicture}[scale=1]
\draw[fill=gray!20!white, opacity=0.3] (-1.8,-1.8) rectangle (2.2,2.2);
\draw[fill=gray!20!white, opacity=0.3] (-1.4,-1.4) rectangle (2.6,2.6);
\node (A) at (.2,.2) {$\times$};
\draw (A) node[below right]{$y$};
\node (B) at (.4,.4) {$\times$};
\draw (B) node[above right]{$v_y$};
\foreach\i in {-3,...,7}{
\foreach\j in {6,...,7}{
\draw[color=gray] (\i*.38,\j*.38) rectangle (\i*.38-.38,\j*.38-.38);
\draw[color=gray] (\i*.38-0.38,\j*.38-3.8) rectangle (\i*.38-.76,\j*.38-3.8-.38);
}}
\foreach\i in {-4,...,-3}{
\foreach\j in {-4,...,6}{
\draw[color=gray] (\i*.38,\j*.38) rectangle (\i*.38-.38,\j*.38-.38);
\draw[color=gray] (\i*.38+3.8-0.38,\j*.38+0.38) rectangle (\i*.38+3.8,\j*.38);
}}
\draw[thick] (-1.8,-1.8) rectangle (2.2,2.2);
\draw[thick] (-1.4,-1.4) rectangle (2.6,2.6);
\end{tikzpicture}
\caption[Covering the set $B(y,R)\Delta B(v_y,R)$ by cubes of radius $r_\Gamma$]{Covering the set $B(y,R)\Delta B(v_y,R)$ by cubes of radius $r_\Gamma$.}\label{DiffCub}
\end{center}
\end{figure}
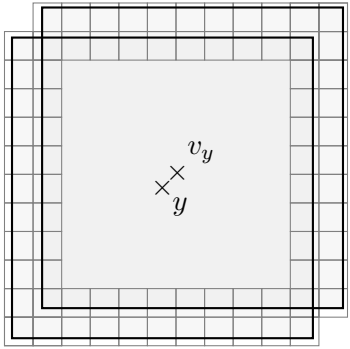

As $\mathcal N_\varep$ is relatively dense, for every $y\in \R^n$, there exists $v_y\in -\mathcal N_\varep$ such that $d_\infty(y, v_y)<R_{\mathcal N_\varep}$. This $v_y$ is the vector $v$ we look for to have the property of Definition~\ref{wap}. Indeed, by triangle inequality, for every $R\ge R_\varep$, we have
\begin{align}\label{EqTriangleIn}
\card\Big(\big(B(0,R)\cap\Gamma\big) & \Delta \big((B(y,R)\cap\Gamma)-v_y\big)\Big)\nonumber\\
         \le & \card\Big(\big(B(0,R)\cap\Gamma\big) \Delta\big((B(v_y,R)\cap\Gamma) - v_y\big)\Big)\\
				     & + \card\Big(\big(B(v_y,R)\cap\Gamma\big) \Delta\big(B(y,R)\cap\Gamma\big)\Big).\nonumber
\end{align}
By Equation~\eqref{eq2.1}, the first term of the right side of the inequality is smaller than $\varep\Vol(B_R)$. It remains to bound the second one.

For every $y\in \R^n$, as $d_\infty(y, v_y)<R_{\mathcal N_\varep}$, the set $B(y,R)\Delta B(v_y,R)$ is covered by
\[\frac{2n(R+r_\Gamma)^{n-1}(R_{\mathcal N_\varep} + r_\Gamma)}{r_\Gamma^n}\]
disjoint cubes of radius $r_\Gamma$ (see Figure \ref{DiffCub}). Thus, as each one of these cubes contains at most one point of $\Gamma$, this implies that
\[\card\Big(\big(B(y,R)\Delta B(v_y,R)\big)\cap \Gamma\Big) \le 2n\frac{(R+r_\Gamma)^{n-1}(R_{\mathcal N_\varep} + r_\Gamma)}{r_\Gamma^n}.\]
Increasing $R_\varep$ if necessary, for every $R\ge R_\varep$, we have 
\[2n\frac{(R+r_\Gamma)^{n-1}(R_{\mathcal N_\varep} + r_\Gamma)}{r_\Gamma^n} \le \varep \Vol(B_R),\]
so	
\[\card\Big(\big(B(y,R)\Delta B(v_y,R)\big)\cap \Gamma\Big) \le \varep \Vol(B_R).\]
This bounds the second term of Equation~\eqref{EqTriangleIn}. We finally get
\[\card\Big(\big(B(0,R)\cap\Gamma\big) \Delta\big(B(y,R)\cap\Gamma-v_y\big)\Big) \le 2 \varep \Vol(B_R),\]
which proves the proposition.
\end{proof}

We now state an easy lemma which asserts that for $\varep$ small enough, the set of translations $\mathcal N_\varep$ is ``stable under additions with a small number of terms''. We will use this lemma in the part concerning discretizations of linear maps.

\begin{lemme}\label{arithProg}
Let $\Gamma$ be an almost periodic pattern, $\varep>0$ and $\ell\in\N$. Then if we set $\varep'=\varep/\ell$ and denote by $\mathcal N_{\varep'}$ the set of translations of $\Gamma$ and $R_{\varep'}>0$ the corresponding radius for the parameter $\varep'$, then for every $k\in\llbracket 1,\ell \rrbracket$ and every $v_1,\cdots,v_\ell\in\mathcal N_{\varep'}$, we have
\[\forall R\ge R_{\varep'},\  D_R^+\Big( \big(\Gamma+\sum_{i=1}^\ell v_i\big)\Delta \Gamma \Big) <\varep.\]
\end{lemme}

\begin{proof}[Proof of Lemma \ref{arithProg}]
Let $\Gamma$ be an almost periodic pattern, $\varep>0$, $\ell\in\N$, $R_0>0$ and $\varep'=\varep/\ell$. Then there exists $R_{\varep'}>0$ such that
\[\forall R\ge R_{\varep'},\  \forall v\in\mathcal N_{\varep'},\  D_R^+\big( (\Gamma+v)\Delta \Gamma \big) <\varep'.\]
We then take $1\le k\le\ell$, $v_1,\cdots,v_k\in\mathcal N_{\varep'}$ and compute
\begin{align*}
D_R^+\Big( \big(\Gamma+\sum_{i=1}^k v_i\big)\Delta \Gamma \Big) & \le \sum_{m=1}^k D_R^+\Big( \big(\Gamma+\sum_{i=1}^m v_i\big)\Delta \big(\Gamma+\sum_{i=1}^{m-1} v_i\big) \Big)\\
             & \le \sum_{m=1}^k D_R^+\Big( \big((\Gamma+v_m)\Delta \Gamma\big) + \sum_{i=1}^{m-1} v_i \Big).
\end{align*}
By the invariance under translation of $D_R^+$, we deduce that
\begin{align*}
D_R^+\Big( \big(\Gamma+\sum_{i=1}^k v_i\big)\Delta \Gamma \Big) & \le \sum_{m=1}^k D_R^+ \big((\Gamma+v_m)\Delta \Gamma\big)\\
						 & \le k\varep'.
\end{align*}
As $k\le \ell$, this ends the proof.
\end{proof}

\begin{rem}\label{arithProg2}
In particular, this lemma implies that the set $\mathcal N_\varep$ contains arbitrarily large patches of lattices of $\R^n$: for every almost periodic pattern $\Gamma$, $\varep>0$ and $\ell\in\N$, there exists $\varep'>0$ such that for every $k_i \in \llbracket -\ell,\ell\rrbracket$ and every $v_1,\cdots,v_n\in\mathcal N_{\varep'}$, we have
\[\forall R\ge R_{\varep'},\  D_{R}^+\Big( \big(\Gamma+\sum_{i=1}^n k_iv_i \big)\Delta \Gamma \Big) <\varep.\]
\end{rem}

We end this section by a proposition stating that the set of $\varep$-almost-periods of an almost periodic pattern possess some almost periodicity.

\begin{prop}\label{azur}
For every $\varep>0$, the set $\mathcal N_\varep$ of $\varep$-almost-periods of an almost periodic pattern $\Gamma$ (see Definition~\ref{DefAlmPer}) contains a relatively dense almost periodic pattern.
\end{prop}

\begin{proof}[Proof of Proposition~\ref{azur}]
Let $\Gamma$ be an almost periodic pattern. Then, by definition, for every $\varep>0$ we can choose $R_\varep>0$ such that the set $\mathcal N_\varep^{R_\varep}$ (defined by Equation~\eqref{EqAlmPer}) is relatively dense.

Consider the map
\[f : \varep \mapsto \limsup_{R\to +\infty} \frac{\card\big(B_R\cap \mathcal N_\varep\big)}{\Vol(B_R)}.\]
As the sets $\mathcal N_\varep$ are decreasing for inclusion, the function $f$ is increasing. Without loss of generality (for example by applying a suitable homothety to the set $\Gamma$), we can suppose that $f(1)=1$. Thus, if we set
\[\mathcal E_\delta = \big\{\varep\in[0,1] \mid \forall \varep'\in[\varep-\delta^2,\varep+\delta^2], |f(\varep)-f(\varep')| <\delta \big\},\]
then an easy calculation shows that the set defined by
\[\mathcal E = \bigcup_{M>0}\bigcap_{m\ge M}\mathcal E_{2^{-m}},\]
satisfies $\Leb(\mathcal E) = 1$. In particular, $\mathcal E$ is dense in $[0,1]$.
\bigskip

Let $\varep\in\mathcal E$ such that $R_\varep$ is locally constant around  $\varep$, $\delta>0$ and $R >0$. Taking a smaller $\delta$ if necessary, we can suppose that $\varep\in\mathcal E_\delta$. We will show that $\mathcal N_{\delta^2/2}$ is a set of $\delta$-translations of the set $\mathcal N_{\varep}$.

Let $w\in\mathcal N_{\delta^2/2}$, and denote
\[A = \frac{\card\Big(\big((\mathcal N_\varep + w)\Delta \mathcal N_\varep\big) \cap B(x,R)\Big)}{\Vol(B_R)},\]
then, we have:
\begin{align*}
A = & \frac{\card\Big(\big((\mathcal N_\varep + w)\setminus \mathcal N_\varep\big) \cap B(x,R)\Big)}{\Vol(B_R)} + \frac{\card\Big(\big(\mathcal N_\varep\setminus (\mathcal N_\varep + w)\big) \cap B(x,R)\Big)}{\Vol(B_R)}\\
  = & \frac{\card\Big(\big((\mathcal N_\varep + w)\setminus \mathcal N_\varep\big) \cap B(x,R)\Big)}{\Vol(B_R)} + \frac{\card\Big(\big((\mathcal N_\varep-w)\setminus \mathcal N_\varep\big) \cap B(x-w,R)\Big)}{\Vol(B_R)}.
\end{align*}
As $w\in\mathcal N_{\delta^2/2}$, and so $-w$ (by Definition~\ref{DefAlmPer}, the sets $\mathcal N_*$ are symmetric), we have $\mathcal N_\varep \pm w \subset \mathcal N_{\varep+\delta^2/2}$ (see Lemma~\ref{arithProg}). Thus,
\begin{equation}\label{label0}
A \le \frac{\card\Big(\big(\mathcal N_{\varep+\delta^2/2}\setminus \mathcal N_\varep\big) \cap B(x,R)\Big)}{\Vol(B_R)} + \frac{\card\Big(\big(\mathcal N_{\varep+\delta^2/2}\setminus \mathcal N_\varep\big) \cap B(x-w,R)\Big)}{\Vol(B_R)}.
\end{equation}

To prove the proposition, it suffices to prove that the first term of the previous bound, denoted by
\[B = \frac{\card\Big(\big(\mathcal N_{\varep+\delta^2/2}\setminus \mathcal N_\varep\big) \cap B(x,R)\Big)}{\Vol(B_R)},\]
is smaller than $3\delta$ for every $x\in \R^n$. As $\mathcal N_{\delta^2/2}$ is relatively dense, there exists $x'\in \mathcal N_{\delta^2/2}$ such that $\|x-x'\| \le R_{\mathcal N_{\delta^2/2}}$ (where $R_{\mathcal N_{\delta^2/2}}$ denotes the constant of ``relative denseness'' of $\mathcal N_{\delta^2/2}$). Thus, 
\begin{equation}\label{label1}
B \le \frac{\card\Big(\big(\mathcal N_{\varep+\delta^2/2}\setminus \mathcal N_\varep\big) \cap B(x',R)\Big)}{\Vol(B_R)} + \frac{\card\Big(\mathcal N_{\varep+\delta^2/2} \cap \big(B(x',R) \Delta B(x,R)\big)\Big)}{\Vol(B_R)}.
\end{equation}

As $x'\in\mathcal N_{\delta^2/2}$, we have $\mathcal N_{\varep+\delta^2/2} - x' \subset \mathcal N_{\varep+\delta^2}$ and $\mathcal N_{\varep} - x' \supset \mathcal N_{\varep-\delta^2/2}$. Then, the first term of the previous bound \eqref{label1} satisfies 
\[\frac{\card\Big(\big(\mathcal N_{\varep+\delta^2/2}\setminus \mathcal N_\varep\big) \cap B(x',R)\Big)}{\Vol(B_R)} \le \frac{\card\Big(\big(\mathcal N_{\varep+\delta^2}\setminus \mathcal N_{\varep-\delta^2/2}\big) \cap B(0,R)\Big)}{\Vol(B_R)}.\]
As $\varep\in\mathcal E_\delta$, we have $f(\varep+\delta^2)-f(\varep-\delta^2/2) \le \delta$. Thus, there exists $R_0>0$ (that depends only on $\delta$) such that for every $R>R_0$, we have  
\begin{equation}\label{label2}
\frac{\card\Big(\big(\mathcal N_{\varep+\delta^2/2}\setminus \mathcal N_\varep\big) \cap B(x',R)\Big)}{\Vol(B_R)} \le 2\delta.
\end{equation}

For its part, the second term of Equation~\eqref{label1} can be bounded by applying the same strategy as in the proof of Proposition \ref{WeakAlmPer} (see also Figure~\ref{DiffCub}): there exists $R_1>0$ (which depends only on the constant of ``uniform discreteness'' of $\mathcal N_{\varep+\delta^2/2}$ and of $\delta$) such that for every $R\ge R_1$, we have
\begin{equation}\label{label3}
\frac{\card\Big(\mathcal N_{\varep+\delta^2/2} \cap \big(B(x',R) \Delta B(x,R)\big)\Big)}{\Vol(B_R)} \le \delta.
\end{equation}

Finally, combining Equations~\eqref{label1}, \eqref{label2} and \eqref{label3}, we get that for every $x\in\R^n$ and every $R\ge \max(R_0,R_1)$ (the maximum depending only on $\delta$),
\[\frac{\card\Big(\big(\mathcal N_{\varep+\delta^2/2}\setminus \mathcal N_\varep\big) \cap B(x,R)\Big)}{\Vol(B_R)}\le 3\delta,\]
which proves that (by Equation~\eqref{label0})
\[\frac{\card\Big(\big((\mathcal N_\varep + w)\Delta \mathcal N_\varep\big) \cap B(x,R)\Big)}{\Vol(B_R)} \le 6\delta.\]
\bigskip

To conclude, there exists a dense subset $\mathcal E$ of $[0,1]$ such that for every $\varep\in\mathcal E$, the set $\mathcal N_\varep$ is an almost periodic pattern. As the sets $\mathcal N_\varep$ are decreasing for inclusion, and each one is relatively dense, we obtain the conclusion of the proposition.
\end{proof}

\section{Weakly almost periodic sets possess a density}

In this section we show that weakly almost periodic sets have a regular enough behaviour at the infinity to possess a density.

\begin{definition}
For a discrete set $\Gamma\subset \R^n$ and $R\ge 1$, we recall that the uniform $R$-density is defined as\index{$D_R^+$}
\[D_R^+(\Gamma) = \sup_{x\in\R^n} \frac{\card\big(B(x,R)\cap \Gamma\big)}{\Vol(B_R)};\]
the uniform upper density is\index{$D^+$}
\[D^+(\Gamma) = \limsup_{R\to +\infty} D_R^+(\Gamma).\]
\end{definition}

Remark that if $\Gamma\subset \R^n$ is Delone for the parameters $r_\Gamma$ and $R_\Gamma$, then its upper density satisfies:
\[\frac{1}{(2R_\Gamma+1)^n} \le D^+(\Gamma) \le \frac{1}{(2r_\Gamma+1)^n}.\]

\begin{prop}\label{limitexist}
Every weakly almost periodic set possesses a uniform density (see Definition~\ref{defunif}).
\end{prop}

We have defined the concept of weakly almost periodic set because it seemed to us that it was the weakest to imply the existence of a uniform density.

%

\begin{rem}\label{Jordan}
The same proof also shows that the same property holds if instead of considering the density $D^+$, we take a \emph{Jordan-measurable}\footnote{We say that a set $J$ is Jordan-measurable if for every $\varep>0$, there exists $\eta>0$ such that there exists two disjoint unions $\mathcal C$ and $\mathcal C'$ of cubes of radius $\eta$, such that $\mathcal C\subset J\subset\mathcal C'$, and that $\Leb(\mathcal C'\setminus\mathcal C)<\varep$.} set $J$ and consider the density $D_J^+(\Gamma)$ of a set $\Gamma\subset \Z^n$ defined by\index{$D_A^+$}
\[D_J^+(\Gamma) = \limsup_{R\to +\infty} \sup_{x\in\R^n} \frac{\card\big(J_R\cap \Gamma\big)}{\Vol(J_R)},\]
where $J_R$ denotes the set of points $x\in\R^n$ such that $R^{-1}x\in J$.
\end{rem}

\begin{proof}[Proof of Proposition \ref{limitexist}]
Let $\Gamma$ be a weakly almost periodic set and $\varep>0$. Then, by definition, there exists $R>0$ such that for all $x,y\in\R^n$, there exists $v\in \R^n$ such that Equation \eqref{EqWeakAlmPer} holds. We take a ``big'' $M\in\R$, $x\in\R^n$ and $R'\ge MR$. We use the tiling of $\R^n$ by the collection of squares $\{B(Ru,R)\}_{u\in (2\Z)^n}$ and the Equation \eqref{EqWeakAlmPer} (applied to the radius $R'$ and the points $0$ and $Ru$) to find the number of points of $\Gamma$ that belong to $B(x,R')$: as $B(x,R')$ contains at least $\lfloor M\rfloor^n$ disjoint cubes $B(Ru,R)$ and is covered by at most $\lceil M\rceil^n$ such cubes, we get (recall that $B_R = B(0,R)$)
\begin{flalign*}
\frac{\lfloor M\rfloor^n \big(\card (B_R\cap\Gamma)-2\varep\Vol(B_R)\big)}{\lceil M\rceil^n\Vol(B_R)} \le \frac{\card\big(B(x,R')\cap \Gamma\big)}{\Vol(B_R)} \le & & &
\end{flalign*}
\begin{flalign*}
 & & & \frac{\lceil M\rceil^n \big(\card (B_R\cap\Gamma)+2\varep\card (B_R \cap \Z^n)\big)}{\lfloor M\rfloor^n\Vol(B_R)},
\end{flalign*}
thus
\begin{flalign*}
\frac{\lfloor M\rfloor^n}{\lceil M\rceil^n}\left(\frac{\card (B_R\cap\Gamma)}{\Vol(B_R)}-2\varep\right) \le \frac{\card\big(B(x,R')\cap \Gamma\big)}{\Vol(B_R)} \le & & &
\end{flalign*}
\begin{flalign*}
 & & & \qquad \qquad \qquad \qquad \frac{\lceil M\rceil^n}{\lfloor M\rfloor^n}\left(\frac{\card (B_R\cap\Gamma)}{\Vol(B_R)}+2\varep\right).
\end{flalign*}
For $M$ large enough, this ensures that for every $R'\ge MR$ and every $x\in \R^n$, the density
\[\frac{\card\big(B(x,R')\cap \Gamma\big)}{\Vol(B_R)}\quad  \text{is close to}  \quad \frac{\card (B_R\cap\Gamma)}{\Vol(B_R)};\]
this finishes the proof of the proposition.
\end{proof}

\section{Discretizations of linear maps}

In this section, we prove that the notions of almost periodic pattern and model set are adapted to discretizations of linear maps.

For the case of model sets, we have the following (trivial) result.

\begin{prop}\label{ImgModel}
Let $\Gamma$ be a model set and $A \in GL_n(\R)$ be an invertible map. Then the set $\widehat A (\Gamma)$ (see Definition~\ref{DefDiscrLin}) is a model set.
\end{prop}

\begin{proof}[Proof of Proposition~\ref{ImgModel}]
Let $\Gamma$ be a model set modelled on a lattice $\Lambda\subset \R^{m+n}$ and a window $W\subset\R^m$. Let $B_1\in M_m(\R)$ and $B_2\in M_n(\R)$ such that the lattice $\Lambda$ is spanned by the matrix
\[\begin{pmatrix}
B_1\\B_2
\end{pmatrix}.\]
Then $\widehat A (\Gamma)$ is the model set modelled on the window $W' = W\times ]-\frac 12,\frac 12]^n$ and the lattice spanned by the matrix
\[\begin{pmatrix}
B_1  & \\
AB_2 & -\Id\\
     & \Id
\end{pmatrix}.\]
\end{proof}

In particular, the image of $\Z^n$ by the discretizations of the matrix $A_!,\cdots,A_k$ is the model set modelled on the window $W = ]-\frac 12,\frac 12]^{nk}$ and the lattice spanned by $M_{A_1,\cdots,A_k}\Z^{n(k+1)}$, where
\[ M_{A_1,\cdots,A_k} = \begin{pmatrix}
A_1 & -\Id &        &        & \\
    & A_2  & -\Id   &        & \\
    &      & \ddots & \ddots & \\
    &      &        & A_k    & -\Id\\
    &      &        &        & \Id
\end{pmatrix}\in M_{n(k+1)}(\R).\]
\bigskip

Concerning almost periodic patterns, we prove that the image of an almost periodic pattern by the discretization of a linear map is still an almost periodic pattern.	

\begin{theoreme}\label{imgquasi}
Let $\Gamma\subset\Z^n$ be an almost periodic pattern and $A\in GL_n(\R)$. Then the set $\widehat A(\Gamma)$ is an almost periodic pattern.
\end{theoreme}

In particular, for every lattice $\Lambda$ of $\R^n$, the set $\pi(\Lambda)$ is an almost periodic pattern. More generally, given a sequence $(A_k)_{k\ge 1}$ of invertible matrices of $\R^n$, the successive images $(\widehat{A_k}\circ\cdots\circ \widehat{A_1})(\Z^n)$ are almost periodic patterns. See Figure~\ref{ImagesSuitesMat} for an example of the successive images of $\Z^2$ by a random sequence of bounded matrices of $SL_2(\R)$.

\begin{figure}[t]
\begin{minipage}[c]{.33\linewidth}
	\includegraphics[width=\linewidth, trim = 1.5cm .5cm 1.5cm .5cm,clip]{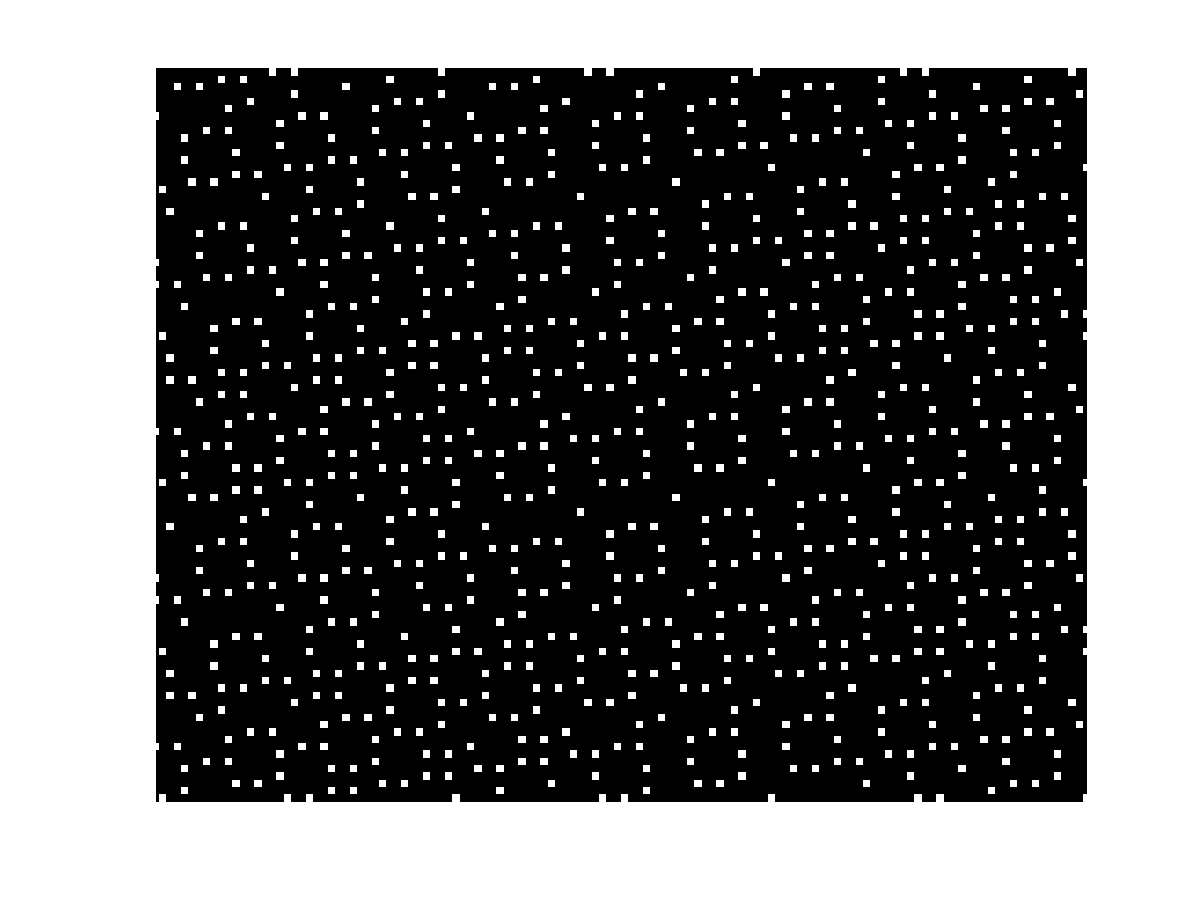}
\end{minipage}\hfill
\begin{minipage}[c]{.33\linewidth}
	\includegraphics[width=\linewidth, trim = 1.5cm .5cm 1.5cm .5cm,clip]{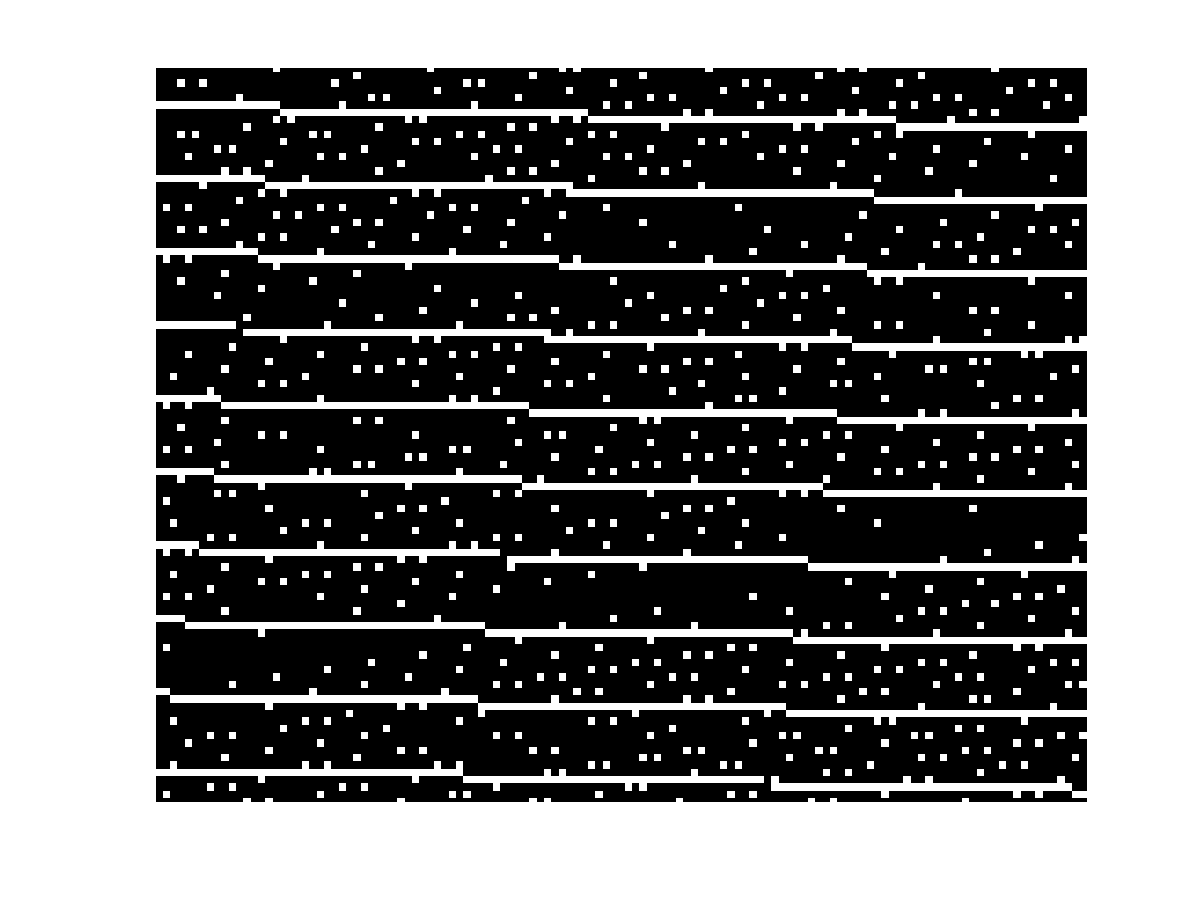}
\end{minipage}\hfill
\begin{minipage}[c]{.33\linewidth}
	\includegraphics[width=\linewidth, trim = 1.5cm .5cm 1.5cm .5cm,clip]{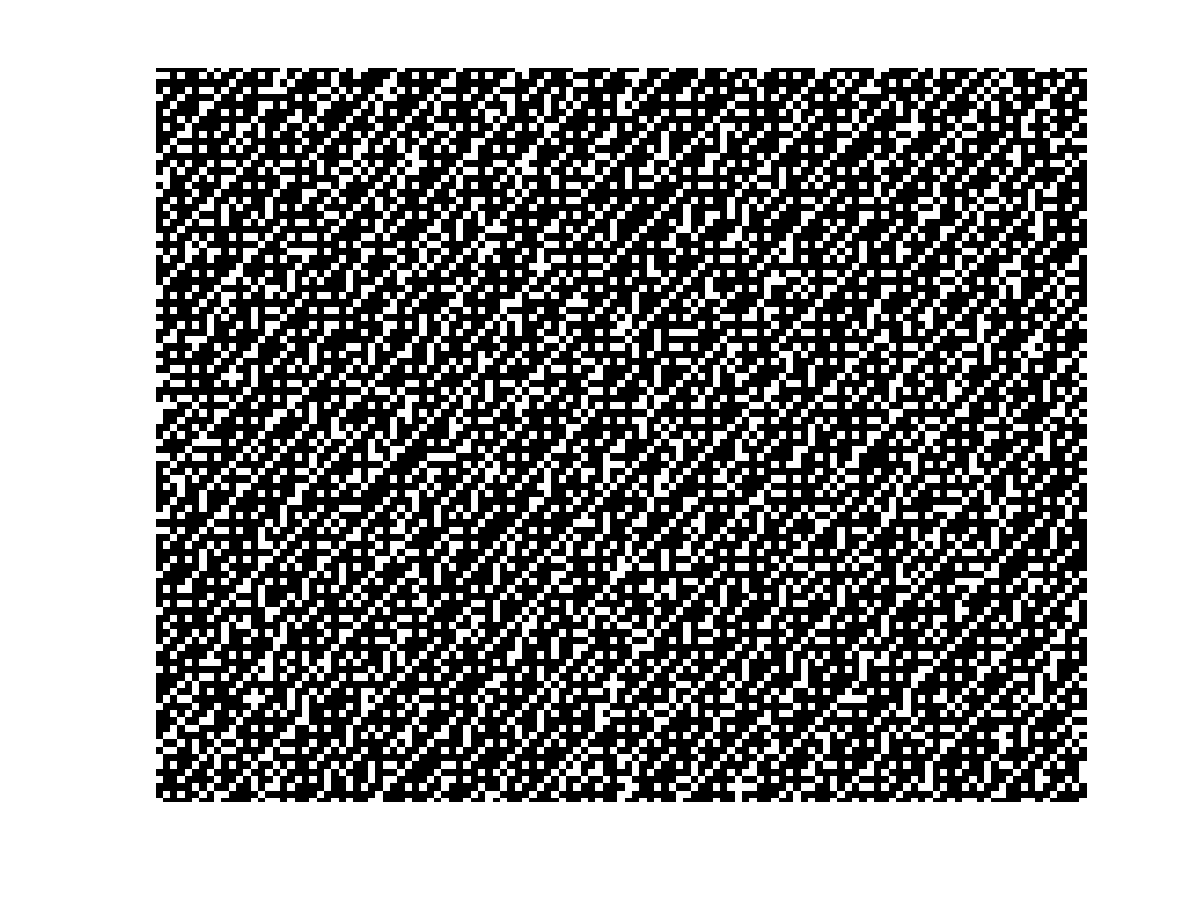}
\end{minipage}

\begin{minipage}[c]{.33\linewidth}
	\includegraphics[width=\linewidth, trim = 1.5cm .5cm 1.5cm .5cm,clip]{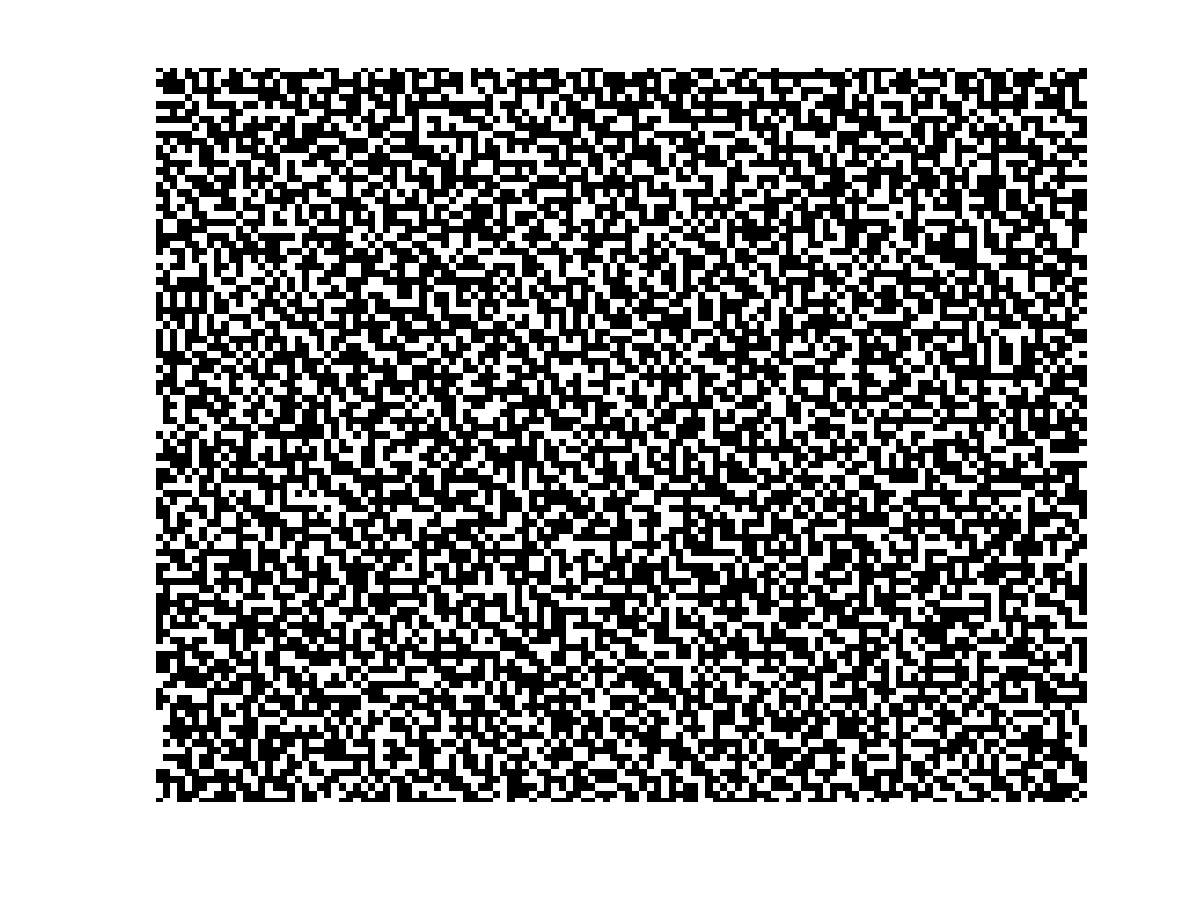}
\end{minipage}\hfill
\begin{minipage}[c]{.33\linewidth}
	\includegraphics[width=\linewidth, trim = 1.5cm .5cm 1.5cm .5cm,clip]{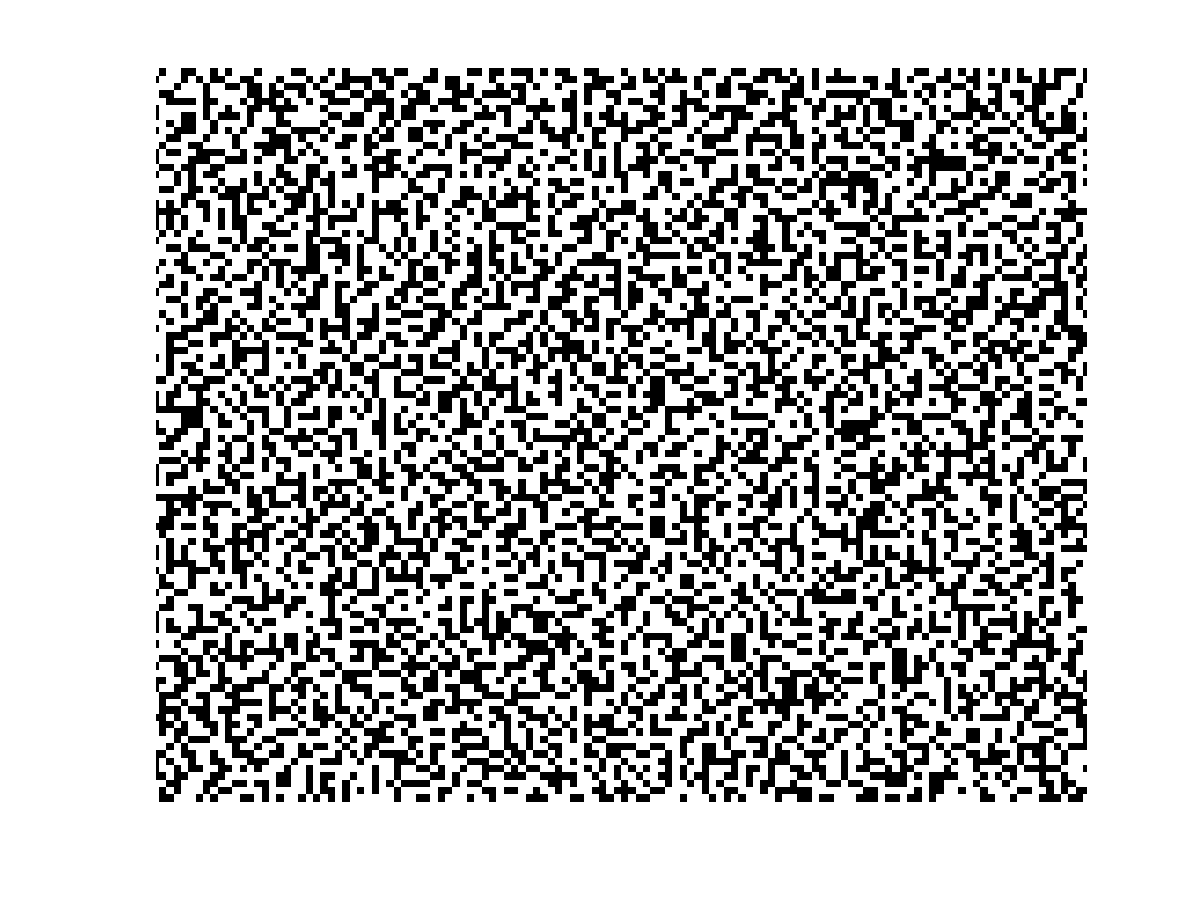}
\end{minipage}\hfill
\begin{minipage}[c]{.33\linewidth}
	\includegraphics[width=\linewidth, trim = 1.5cm .5cm 1.5cm .5cm,clip]{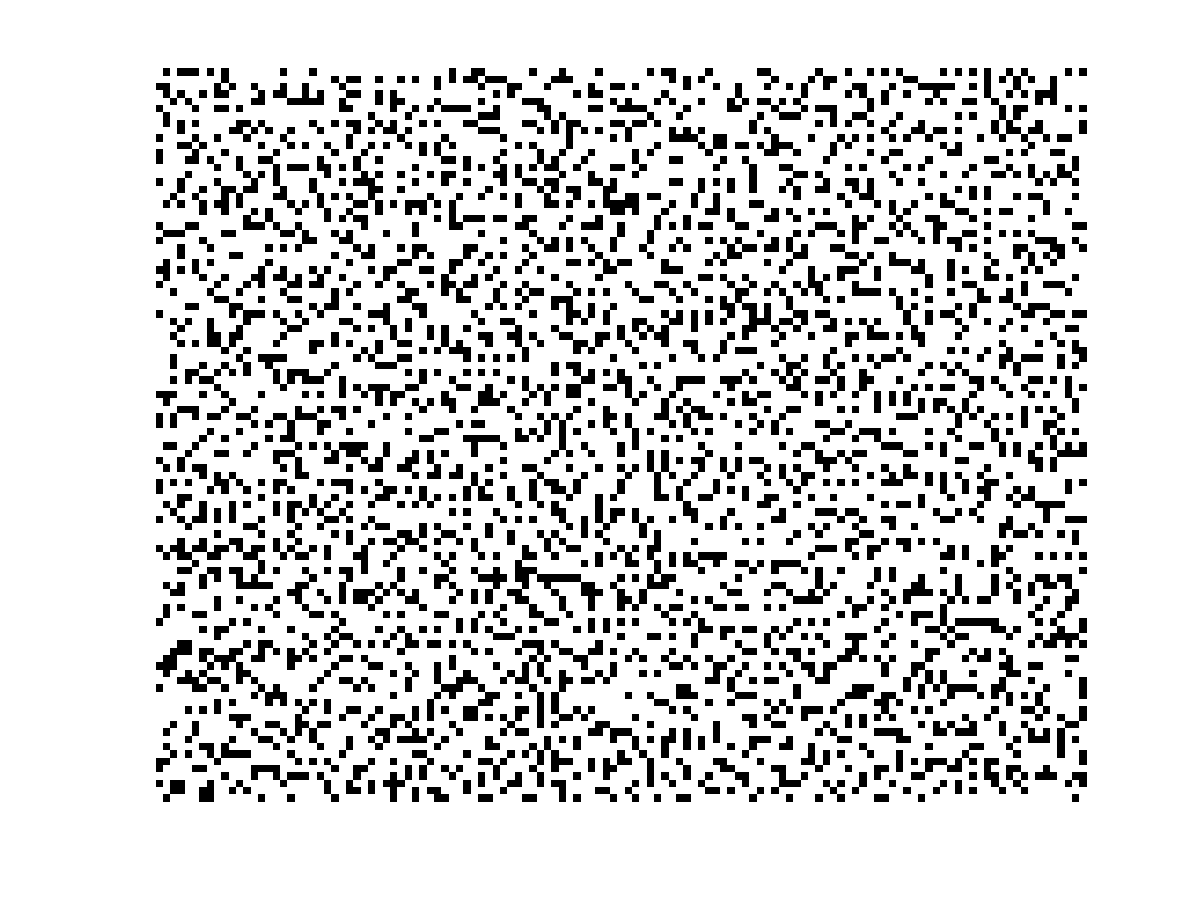}
\end{minipage}
\caption[Successive images of $\Z^2$ by discretizations of random matrices]{Successive images of $\Z^2$ by discretizations of random matrices in $SL_2(\R)$, a point is black if it belongs to $(\widehat{A_k}\circ\cdots\circ\widehat{A_1})(\Z^2)$. The $A_i$ are chosen randomly, using the singular value decomposition: they are chosen among the matrices of the form $R_\theta D_t R_{\theta'}$, with $R_\theta$ the rotation of angle $\theta$ and $D_t$ the diagonal matrix $\operatorname{Diag}(e^t,e^{-t})$, the $\theta$, $\theta'$ being chosen uniformly in $[0,2\pi]$ and $t$ uniformly in $[-1/2,1/2]$. From left to right and top to bottom, $k=1,\, 2,\, 3,\, 5,\, 10,\, 20$.}\label{ImagesSuitesMat}
\end{figure}

\begin{notation}\label{intelligent}
For $A\in GL_n(\R)$, we denote $A=(a_{i,j})_{i,j}$. We denote by $I_\Q(A)$\index{$I_\Q(A)$} the set of indices $i$ such that $a_{i,j}\in\Q$ for every $j\in\llbracket 1,n\rrbracket$
\end{notation}

The proof of Theorem~\ref{imgquasi} relies on the following remark:

\begin{rem}\label{pmtriv}
If $a\in\Q$, then there exists $q\in\N^*$ such that $\{ax\mid x\in\Z\}\subset \frac{1}{q}\Z$. On the contrary, if $a\in\R\setminus\Q$, then the set $\{ax\mid x\in\Z\}$ is equidistributed in $\R/\Z$.
\end{rem}

Thus, in the rational case, the proof will lie in an argument of periodicity. On the contrary, in the irrational case, the image $A(\Z^n)$ is equidistributed modulo $\Z^n$: on every large enough domain, the density does not move a lot when we perturb the image set $A(\Z^n)$ by small translations. This reasoning is formalized by Lemmas~\ref{tiroir} and \ref{équi}. 

More precisely, for $R$ large enough, we would like to find vectors $w$ such that $D^+_R\big((\pi(A\Gamma) +w)\Delta \pi(A\Gamma)\big)$ is small. We know that there exists vectors $v$ such that $D^+_R\big((\Gamma+v)\Delta\Gamma\big)$ is small; this implies that $D^+_R\big((A\Gamma+Av)\Delta A\Gamma\big)$ is small, thus that $D^+_R\big(\pi(A\Gamma+Av)\Delta \pi(A\Gamma)\big)$ is small. The problem is that in general, we do not have $\pi(A\Gamma+Av) = \pi(A\Gamma)+\pi(Av)$. However, this is true if we have $Av\in\Z^n$. Lemma~\ref{tiroir} shows that in fact, it is possible to suppose that $Av$ ``almost'' belongs to $\Z^n$, and Lemma~\ref{équi} asserts that this property is sufficient to conclude.

The first lemma is a consequence of the pigeonhole principle.

\begin{lemme}\label{tiroir}
Let $\Gamma\subset \Z^n$ be an almost periodic pattern, $\varep>0$, $\delta>0$ and $A\in GL_n(\R)$. Then we can suppose that the elements of $A(\mathcal N_\varep)$ are $\delta$-close to $\Z^n$. More precisely, there exists $R_{\varep,\delta}>0$ and a relatively dense set $\widetilde{\mathcal N}_{\varep,\delta}$\index{$\widetilde{\mathcal N}_{\varep,\delta}$} such that 
\[\forall R\ge R_{\varep,\delta},\  \forall v\in\widetilde{\mathcal N}_{\varep,\delta},\  D_R^+\big( (\Gamma+v)\Delta \Gamma \big) <\varep,\]
and that for every $v\in\widetilde{\mathcal N}_{\varep,\delta}$, we have $d_\infty(Av,\Z^n)<\delta$. Moreover, we can suppose that for every $i\in I_\Q(A)$ and every $v\in\widetilde{\mathcal N}_{\varep,\delta}$, we have $(Av)_i\in \Z$.
\end{lemme}

The second lemma states that in the irrational case, we have continuity of the density under perturbations by translations.

\begin{lemme}\label{équi}
Let $\varep>0$ and $A\in GL_n(\R)$. Then there exists $\delta>0$ and $R_0>0$ such that for all $w\in B_\infty(0,\delta)$ (such that for every $i\in I_\Q(A)$, $w_i=0$), and for all $R\ge R_0$, we have
\[D_R^+\big(\pi(A\Z^n) \Delta \pi(A\Z^n+w) \big) \le \varep.\]
\end{lemme}

\begin{rem}
The assumption ``for every $i\in I_\Q(A)$, $v_i=0$'' is necessary to obtain the conclusion of the lemma (see \cite[Section 8.1]{Guih-These}).
\end{rem}

We begin by the proofs of both lemmas, and prove Theorem~\ref{imgquasi} thereafter.

\begin{proof}[Proof of Lemma \ref{tiroir}]
Let us begin by giving the main ideas of the proof of this lemma. For $R_0$ large enough, the set of remainders modulo $\Z^n$ of vectors $Av$, where $v$ is a $\varep$-translation of $\Gamma$ belonging to $B_{R_0}$, is close to the set of remainders modulo $\Z^n$ of vectors $Av$, where $v$ is any $\varep$-translation of $\Gamma$. Moreover (by the pigeonhole principle), there exists an integer $k_0$ such that for each $\varep$-translation $v\in B_{R_0}$, there exists $k\le k_0$ such that $A(k v)$ is close to $\Z^n$. Thus, for every $\varep$-translation $v$ of $\Gamma$, there exists a $(k_0-1)\varep$-translation $v' = (k-1)v$, belonging to $B_{k_0 R_0}$, such that $A(v+v')$ is close to $\Z^n$. The vector $v+v'$ is then a $k_0\varep$-translation of $\Gamma$ (by additivity of the translations) whose distance to $v$ is smaller than $k_0 R_0$.
\bigskip

We now formalize these remarks. Let $\Gamma$ be an almost periodic pattern, $\varep>0$ and $A\in GL_n(\R)$. First of all, we apply the pigeonhole principle. We partition the torus $\R^n/\Z^n$ into squares whose sides are smaller than $\delta$; we can suppose that there are at most  $\lceil 1/\delta\rceil^n$ such squares. For $v\in \R^n$, we consider the family of vectors $\{A(kv)\}_{0\le k\le \lceil 1/\delta\rceil^n}$ modulo $\Z^n$. By the pigeonhole principle, at least two of these vectors, say $A(k_1v)$ and $A(k_2v)$, with $k_1<k_2$, lie in the same small square of $\R^n/\Z^n$. Thus, if we set $k_v = k_2-k_1$ and $\ell = \lceil 1/\delta\rceil^n$, we have
\begin{equation}\label{eqdistZ}
1\le k_v\le \ell \quad \text{and} \quad d_\infty\big(A(k_vv),\Z^n\big)\le\delta.
\end{equation}
To obtain the conclusion in the rational case, we suppose in addition that $v\in q\Z^n$, where $q\in\N^*$ is such that for every $i\in I_\Q(A)$ and every $1\le j\le n$, we have $q\, a_{i,j}\in\Z$ (which is possible by Remark~\ref{arithProg2}).

We set $\varep'=\varep/\ell$. By the definition of an almost periodic pattern, there exists $R_{\varep'}>0$ and a relatively dense set ${\mathcal N}_{\varep'}$ such that Equation \eqref{EqAlmPer} holds for the parameter $\varep'$:
\begin{equation}\label{EqAlmPer3}\tag{\ref{EqAlmPer}'}
\forall R\ge R_{\varep'},\  \forall v\in\mathcal N_{\varep'},\  D_R^+\big( (\Gamma+v)\Delta \Gamma \big) <\varep',
\end{equation}

We now set
\[P = \big\{Av\operatorname{mod} \Z^n \mid v\in {\mathcal N}_{\varep'}\big\} \quad \text{and} \quad P_R = \big\{Av\operatorname{mod} \Z^n \mid v\in \mathcal N_{\varep'}\cap B_R\big\}.\]
We have $\bigcup_{R>0} P_R = P$, so there exists $R_0>R_{\varep'}$ such that $d_H(P,P_{R_0})<\delta$ (where $d_H$\index{$d_H$} denotes Hausdorff distance). Thus, for every $v\in\mathcal N_{\varep'}$, there exists $v'\in \mathcal N_{\varep'}\cap B_{R_0}$ such that
\begin{equation}\label{eq666}
d_\infty(Av-Av',\Z^n)<\delta.
\end{equation}

We then remark that for every $v'\in {\mathcal N}_{\varep'}\cap B_{R_0}$, if we set $v'' = (k_{v'}-1)v'$, then by Equation \eqref{eqdistZ}, we have
\[d_\infty(Av' + Av'',\Z^n) = d_\infty\big(A(k_{v'}v'),\Z^n\big)\le\delta.\]
Combining this with Equation~\eqref{eq666}, we get
\[d_\infty(Av + Av'',\Z^n)\le 2\delta,\]
with $v''\in B_{\ell R_0}$. 

On the other hand, $k_{v'}\le \ell$ and Equation \eqref{EqAlmPer3} holds, so Lemma \ref{arithProg} (more precisely, Remark~\ref{arithProg2}) implies that $v''\in \mathcal N_\varep$, that is
\[\forall R\ge R_{\varep'},\  D_{R}^+\big( (\Gamma+ v'')\Delta \Gamma \big) <\varep.\]

In other words, for every $v\in\mathcal N_{\varep'}$, there exists $v''\in \mathcal N_\varep \cap B_{\ell R_0}$ (with $\ell$ and $R_0$ independent from $v$) such that $d_\infty\big(A(v+v''),\Z^n\big)<2\delta$. The set $\widetilde{\mathcal N}_{2\varep,2\delta}$ we look for is then the set of such sums $v+v''$.
\end{proof}

\begin{proof}[Proof of Lemma \ref{équi}]
Under the hypothesis of the lemma, for every $i\notin I_\Q(A)$, the sets
\[\left\{\sum_{j=1}^n a_{i,j} x_j\mid (x_j)\in\Z^n\right\},\]
are equidistributed modulo $\Z$. Thus, for all $\varep>0$, there exists $R_0>0$ such that for every $R\ge R_0$,
\[D_R^+\big\{v\in\Z^n \,\big|\, \exists i\notin I_\Q(A) : d\big((Av)_i,\Z+\frac12\big)\le \varep\big\} \le 2(n+1)\varep.\]
As a consequence, for all $w\in\R^n$ such that $\|w\|_\infty\le\varep/(2(n+1))$ and that $w_i=0$ for every $i\in I_\Q(A)$, we have
\[D_R^+\big(\pi(A\Z^n) \Delta \pi(A(\Z^n+w))\big)\le\varep.\]
Then, the lemma follows from the fact that there exists $\delta>0$ such that $\|A(w)\|_\infty\le \varep/(2(n+1))$ as soon as $\|w\|\le\delta$.
\end{proof}

\begin{proof}[Proof of Theorem \ref{imgquasi}]
Let $\varep>0$. Lemma \ref{équi} gives us a corresponding $\delta>0$, that we use to apply Lemma \ref{tiroir} and get a set of translations $\widetilde{\mathcal N}_{\varep,\delta}$. Then, for every $v\in \widetilde{\mathcal N}_{\varep,\delta}$, we write $\pi(Av) = Av + \big(\pi(Av)-Av\big) = Av + w$. The conclusions of Lemma~\ref{tiroir} imply that $\|w\|_\infty <\delta$, and that $w_i=0$ for every $i\in I_\Q(A)$.

We now explain why $\hat Av = \pi(Av)$ is a $\varep$-translation for the set $\widehat A(\Gamma)$. Indeed, for every $R\ge \max(R_{\varep,\delta},MR_0)$, where $M$ is the maximum of the greatest modulus of the eigenvalues of $A$ and of the greatest modulus of the eigenvalues of $A^{-1}$, we have
\begin{align*}
D^+_R \Big(\pi(A\Gamma) \Delta \big(\pi(A \Gamma)+\widehat Av\big)\Big) \le &\ D^+_R \Big(\pi(A\Gamma) \Delta \big(\pi(A \Gamma)+w\big)\Big)\\
                  & + D^+_R \Big(\big(\pi(A\Gamma) + w\big) \Delta \big(\pi(A \Gamma)+\widehat Av\big)\Big)
\end{align*}
(where $w=\pi(Av)-Av$). By Lemma \ref{équi}, the first term is smaller than $\varep$. For its part, the second term is smaller than
\[D^+_R\big((A\Gamma + Av) \Delta A \Gamma\big) \le M^2 D^+_{RM}\big((\Gamma + v) \Delta \Gamma\big),\]
which is smaller than $\varep$ because $v\in\mathcal N_\varep$.
\end{proof}

\bibliographystyle{amsalpha}
\bibliography{../../Biblio}

\end{document}